\documentclass{amsart}
\usepackage{amsfonts,amssymb,amscd,amsmath,enumerate,verbatim,calc}
\usepackage[all]{xy}

\newcommand{\E}{\mathbb{E} }
\newcommand{\n}{\mathfrak{n} }
\newcommand{\m}{\mathfrak{m} }

\newcommand{\p}{\mathfrak{p} }
\newcommand{\bX}{\mathbf{X} }

\newcommand{\Oc}{\mathcal{O} }
\newcommand{\Fc}{\mathcal{F} }
\newcommand{\Gc}{\mathcal{G} }
\newcommand{\Mr}{\mathcal{M} }
\newcommand{\Z}{\mathbb{Z} }
\newcommand{\C}{\mathcal{C} }
\newcommand{\rt}{\rightarrow}

\newcommand{\ov}{\overline}

\newcommand{\image}{\operatorname{image}}
\newcommand{\coker}{\operatorname{coker}}
\newcommand{\Reg}{\operatorname{Reg}}
\newcommand{\gr}{\operatorname{gr}}

\newcommand{\height}{\operatorname{height}}

\newcommand{\Ass}{\operatorname{Ass}}
\newcommand{\degree}{\operatorname{degree}}
\newcommand{\Min}{\operatorname{Min}}
\newcommand{\Spec}{\operatorname{Spec}}
\newcommand{\mSpec}{\operatorname{m-Spec}}
\newcommand{\Supp}{\operatorname{Supp}}
\newcommand{\rad}{\operatorname{rad}}
\newcommand{\charr}{\operatorname{char}}

\newcommand{\Ext}{\operatorname{Ext}}

\theoremstyle{plain}

\newtheorem{theorem}{Theorem}[section]
\newtheorem{corollary}[theorem]{Corollary}
\newtheorem{lemma}[theorem]{Lemma}
\newtheorem{proposition}[theorem]{Proposition}

\theoremstyle{definition}
\newtheorem{definition}[theorem]{Definition}
\newtheorem{s}[theorem]{}
\newtheorem{remark}[theorem]{Remark}
\newtheorem{example}[theorem]{Example}

\theoremstyle{remark}

\begin{document}

\title[Associated primes of Local cohomology modules]{Associated primes of Local cohomology modules over Regular rings}
\author{Tony~J.~Puthenpurakal}
\date{\today}
\address{Department of Mathematics, IIT Bombay, Powai, Mumbai 400 076}

\email{tputhen@math.iitb.ac.in}

 \begin{abstract}
Let $R$ be an excellent regular ring of dimension $d$ containing a field $K$ of characteristic zero. Let $I$ be an ideal in $R$. We show that $\Ass H^{d-1}_I(R)$ is a finite set. As an application we show that if $I$ is an  ideal of height $g$  with $\height Q = g$ for all minimal primes of $I$ then for all but finitely many primes $P \supseteq I$ with $\height P \geq g +2$, the topological space $\Spec^\circ(R_P/IR_P)$ is connected.
\end{abstract}
 \maketitle
\section{introduction}
Throughout this paper $R$ is a commutative Noetherian ring. If $M$ is an $R$-module and if $I$ is an ideal in $R$, we denote by $H^i_I(M)$ the  $i^{th}$ local cohomology module of $M$ with respect to $I$.

The following conjecture  is due to Lyubeznik  \cite{Lyu-3}; 

\s \label{conj}\textbf{Conjecture:} If $R$ is a regular ring, then each local cohomology module $H^i_I(R)$ has
finitely many associated prime ideals.

There are many cases where this conjecture is true: by work of Huneke and Sharp \cite{HuSh},
for regular rings $R$ of prime characteristic; by work of Lyubeznik, for regular local and
affine rings of characteristic zero \cite{Lyu-1},  for unramified regular local rings of mixed
characteristic \cite{Lyu-4}. It is also true for smooth $\mathbb{Z}$-algebras by work of Bhatt et.al \cite{BB}.

 In \cite{Lyu-3} Lyubeznik especially asked whether \ref{conj} is valid for a regular ring $R$ containing a field of characteristic zero. It is easy to give examples where existing techniques, to show finiteness of associate primes of local cohomology modules, fail:
 \begin{example}
 \begin{enumerate}[\rm (a)]
 \item
 Let $(S,\m)$ be a complete local domain of dimension $d \geq 2$ containing a field of characteristic zero. Assume $S$ is not regular. Let the singular locus be defined by the ideal $J$. Notice $J \neq 0$. Let $x \in J$ be non-zero. Set $R = S_x$. Then $R$ is a  domain of dimension $d-1$, see \cite[Lemma 1, p.\ 247]{Mat}. Also clearly $R$ is regular.
 
 \item
 Let $T$ be a regular domain as above  containing a field $K$ of characteristic zero. Let $f \in K[X_1,\ldots, X_n]$ be a smooth polynomial. Then $R = T[X_1,\ldots, X_n]/(f)$ is a regular ring. 
 \end{enumerate}
 In both the examples above we do not know whether $\Ass_R (H^i_I(R))$ is a finite set for all ideals $I$ of $R$.
 \end{example}
In all the essential cases where finiteness of associated primes is known, the local cohomology modules of $R$ have some additional global structure. If $R$ is of characteristic $p$ then local cohomology modules have structure of $F$-modules, \cite{Lyu-2}. In characteristic zero,  for complete local rings  and smooth affine algebras over algebraically closed field, local cohomology modules have an appropriate $D$-module structure. For smooth $\Z$-algebras the authors in \cite{BB} use a rather clever mixture of $D$-module and $F$-module theory. For a general regular ring containing a field of characteristic zero there is no obvious structure that local cohomology modules satisfy that we can exploit to prove finiteness of associate primes. 

For the rest of the paper assume that $R$ contains a field of characteristic zero.
For simplicity we assume that $\dim R = d$ is finite.
By Grothendieck vanishing theorem $H^i_I(R) = 0$ for all $i > d$, see \cite[6.1.2]{BSh}. In general for a Noetherian ring $R$ of dimension $d$
the set $\Ass_R(H^d_I(R))$ is  finite, see  \cite[3.11]{BRS} (also see \cite[2.3]{TM}). If $R$ is a regular ring of dimension $d$ and $I$ is an ideal in $R$ then using the Hartshorne-Lichtenbaum theorem, cf. \cite[14.1]{a7}  it is easy to prove that
$$\Ass_R (H^d_I(R)) = \left\{ P  \left|  P \in \Min R/I \ \text{and} \ \height P = d   \right.   \right \}.  $$ 
The main result of this paper is
\begin{theorem}\label{main-T}
Let $R$ be an excellent regular ring of dimension $d$ containing a field of characteristic zero.  Let $I$ be an ideal in $R$. Then
$\Ass_R (H^{d-1}_I(R))$ is a finite set.
\end{theorem}

The main idea of this  paper is that it is fruitful to look at the following relative situation:
Let $I \supseteq J$ be ideals in a Noetherian ring $R$. We have a natural map
$\theta^i_{I,J} \colon H^i_I(R) \rt H^i_J(R)$ for each $i \geq 0$.  Fix $i_0 \geq 0$.
Set
\[
\C^{i_0}_R =  \left\{ (I,J) \left| I \supset J  \ \text{and} \  \sharp \Ass_R (\image \theta^{i_0}_{I,J}) = \infty  \right.   \right \}.
\]
Here $\sharp S $ denotes the number of elements in a set $S$.  If $ \Ass_R H^{i_0}_I(R)$ is infinite then $(I,I)\in \C^{i_0}_R$. Conversely if $(I,J) \in \C^{i_0}_R$ then $ \Ass_R H^{i_0}_J(R)$ is infinite.
We partially order $\C^{i_0}_R $ as follows: Set $(I,J) \preceq (I^\prime, J^\prime)$ if $I^\prime \supseteq I$ and $J^\prime \supseteq J$. It is easy to see that every ascending chain in $\C^{i_0}_R $ stabilizes.  If  $\C^{i_0}_R$ is non-empty then its maximal elements 
have some peculiar properties, see \ref{pec}.

The following result is a crucial ingredient in the proof of Theorem \ref{main-T}.
\begin{theorem}\label{main} Let $R$ be an excellent regular domain  of dimension $d$ containing an uncountable field of characteristic zero. Assume for some $i_0$ the set $\C^{i_0}_R $ is non-empty. Let $(I,J)$ be a maximal element in $\C^{i_0}_R $. Then there exists a multiplicatively closed set $S$ of $R$ such that in the ring $A = S^{-1}R$ we have
\begin{enumerate}[\rm(a)]
\item
$(S^{-1}I, S^{-1}J)$ is a maximal element in $\C^{i_0}_ A$.
\item
$\height S^{-1}I = i_0$.
\item
$S^{-1}I = P_1\cap P_2 \cap \cdots \cap P_r$ where $P_i$ is a prime in $A$ of height $i_0$.
\item
$\Ass \image \theta^{i_0}_{S^{-1}I,S^{-1}J} \supseteq \mSpec(A)$.
\item
$\height \m = \height \m^\prime$ for $\m , \m^\prime \in \mSpec(A)$.
\item
$\mSpec(A)$ is a countably infinite set.
\end{enumerate}
Furthermore if $\Ass_R (H^r_L(R))$ is a finite set for all $r < i_0$ and for all ideals $L$ of $R$ then    $\height S^{-1}J = i_0 - 1$
\end{theorem}

The assumption that $R$ is a domain and contains an uncountable field are mild hypotheses, see section \ref{flat-s}.  The assumption on excellence of $R$ is  satisfied by most examples. As an easy  consequence of Theorem \ref{main} we get the following significant simplification for Lyubeznik's conjecture.
\begin{corollary}\label{cor-main}
The following are equivalent:
\begin{enumerate}[\rm (i)]
\item
Lyubeznik's conjecture has a positive answer for all excellent regular rings of  dimension $ \leq d$ and containing a field of characteristic zero.
\item
For all excellent regular domains $R$ of  dimension $\leq d$, containing an uncountable  field of characteristic zero,
 $\Ass_R (H^{g+1}_J(R))$ is a finite set for all ideals $J$ of height $g$.
\end{enumerate}
\end{corollary}

As an application of our Theorem \ref{main-T} we get
\begin{corollary}\label{d4}
Let $R$ be an excellent regular ring containing a field $K$ and of dimension $d  \leq 4$. Then for any ideal $I$ we have
$\Ass H^i_I(R)$ is a finite set for all $i \geq 0$.
\end{corollary}

If $M$ is an $R$-module then set 
\[
\Ass_R^i(M) = \{ P \mid P \in \Ass M \ \text{and} \ \height P = i \}.
\]
Using the  Hartshorne-Lichtenbaum theorem we get that if $R$ is regular and $I$ is an ideal in $R$ then
\[
\bigcup_{i \geq 0} \Ass_R^i(H^i_I(R))  = \Min R/I; \quad \text{see \ref{HL-C}.}
\]
As an application of our result we get 
\begin{corollary}\label{main-app}
Let $R$ be an excellent regular ring  of dimension $d$ and  containing a field of characteristic zero.  Let $I$ be an ideal in $R$. Then
\[
\bigcup_{i \geq 0} \Ass_R^{i+1}(H^i_I(R))  \quad \text{is a finite set}.
\]
\end{corollary}
We need a  description of primes which appear in the Corollary \ref{main-app}.
To do it we first make the following definition: Recall if $(A,\m)$ is a local ring 
then $\Spec^\circ(A) = \Spec(A) \setminus \{ \m \}$ considered as a subspace of $\Spec(A)$.
\begin{definition}
Let $(A,\m)$ be a local ring and  let $I$ be an ideal in $A$.  We say $\Spec^\circ(A/I)$ is \textit{absolutely connected} if for every flat local map $(A,\m) \rt (B,\n)$ with $\m B = \n$, $B$ complete and $B/\n$ algebraically closed,  $\Spec^\circ(B/IB)$ is connected 
\end{definition}

\begin{remark}\label{c-abs-int}
It is easy to see that if $\Spec^\circ(A/I)$ is absolutely connected then it is connected.
\end{remark}

As a consequence of a result of Ogus \cite[2.11]{O} (also see \cite[1.1]{HL}), we get that 
\begin{proposition}\label{ogus}
Let $R$ be a regular ring  of dimension $d$ and containing a field of characteristic zero and let $I$ be an ideal in $R$ of height $g$. Assume $\height Q = g$ for all minimal primes $Q$ of $I$.
Then
\[
\bigcup_{i \geq 0} \Ass_R^{i+1}(H^i_I(R)) =  \left\{  P  \left| \begin{cases}
P \supseteq I, \\ \height P \geq g + 2 \ \text{and} \\ \ \Spec^\circ(R_P/I_P) \ \text{is \emph{NOT} absolutely connected} \end{cases}
  \right.   \right \}. 
\]
\end{proposition}
As an immediate application we get 
\begin{corollary}
(with hypotheses as in Proposition \ref{ogus}) For all but finitely many primes $P \supseteq I$ and $\height P \geq g + 2$ we get that $\Spec^\circ(R_P/I_P)$ is  absolutely connected. In particular $\Spec^\circ(R_P/I_P)$ is connected.
\end{corollary}

Here is an overview of the contents of the paper. In section two we discuss a few general results that we need. In the next section we discuss the flat extension $R \rt R[[X]]_X$. We need it as an essential technique in our paper requires an uncountable field contained in $R$. In section four we discuss countable prime avoidance. We also give a construction which is used several times in our paper. 
In section five we prove Theorem \ref{main}. In the next section we prove our main result Theorem \ref{main-T}. In section seven we prove the simplicity of a $D$-module. This is needed in the proof of Theorem \ref{main-T}. In section eight we prove Corollary \ref{d4}. Finally in section nine we  give a proof of Corollary \ref{main-app} and Proposition \ref{ogus}. 
\section{Generalities}
In this section we prove some general results. Some of it are perhaps already known to the experts. However we prove them as we do not have a reference.

We first prove the following general result:
\begin{proposition}\label{count-ass}
Let $R$ be a Noetherian ring and let $I$ be an ideal in $R$. Let $M$ be a finitely generated $R$-module. Then
$\Ass_R (H^i_I(M))$ is a countable set.
\end{proposition}
To prove the Proposition we need the following notion. 
\begin{definition}
We say an $R$-module $E$ is \textit{countably generated} if there exists a countable set of elements $\{ e_n \}_{n \geq 1}$ which 
generate $E$ as a $R$-module.
\end{definition} 
The following Lemma is useful:
\begin{lemma}\label{count-lemm}
Let $R$ be a Noetherian ring and let $E$ be a countably generated $R$-module. Then $\Ass_R(E)$ is a countable set.
\end{lemma}
\begin{proof}
Let  $\{ e_n \}_{n \geq 1}$  
generate $E$ as a $R$-module. For $m \geq 1$, let $D_m $ be the $R$-submodule of $E$ generated by $e_1,\cdots,e_m$. Clearly
$D_m \subseteq D_{m+1}$ for all $m \geq 1$ and $E = \bigcup_{m\geq 1} D_m$. Clearly $\bigcup_{m \geq 1} \Ass_R (D_m) \subseteq \Ass_R (E)$. If $P \in \Ass_R(E)$ then $P = (0 \colon u)$ for some $u \in E$. Say $u \in D_r$. Then $P \in \Ass_R (D_r)$. Thus
 $\bigcup_{m \geq 1} \Ass_R (D_m) = \Ass_R (E)$.
 
Since $R$ is Noetherian and $D_m$ a finitely generated $R$-module;  
$\Ass_R (D_m)$ is a finite set for all $m \geq 1$. It follows that $\Ass_R (E)$ is a countable set.
\end{proof}
We now give
\begin{proof}[Proof of Proposition \ref{count-ass}]
Fix $i \geq 0$. Then $H^i_I(M) = \varinjlim \Ext^i_R(R/I^n, M)$. In particular $H^i_I(M)$ is a quotient of
$\bigoplus_{n \geq 1} \Ext^i_R(R/I^n, M)$. It follows that $H^i_I(M)$ is countably generated. The result now follows from Lemma \ref{count-lemm}.
\end{proof}

\s We need to compare maps $R$-linear maps $f,g \colon E \rt F$ where $E, F$ are $R$-modules. 
\begin{definition}
Let $f,g \colon E \rt F$ be $R$-linear maps. We say $f \cong g$ if there exists isomorphisms $\alpha \colon E \rt E$ and $\beta \colon F \rt F$ such that the following diagram commutes:
\[
\xymatrixrowsep{3pc}
\xymatrixcolsep{2.5pc}
\xymatrix{
     E \ar@{->}[r]^{f}
     \ar@{->}[d]^{\alpha}
&   F    \ar@{->}[d]^{\beta} 
\\
 E \ar@{->}[r]^{g}
&   F  
}
\]
\end{definition}
The following result is clear:
\begin{proposition}\label{cong}
Let $R$ be a ring and let $E, F$ be $R$-modules. Let $f,g \colon E \rt F$ be $R$-linear. If $f \cong g$ then we have the following isomorphisms of $R$-modules:
\[
\ker f \cong \ker g, \quad \image f \cong \image g \quad \text{and} \ \ \coker f \cong \coker g. \qed
\]
\end{proposition}

\s Let $R$ be a Noetherian ring and let $I, J$ be  ideals of $R$ with $I \supseteq J$.
Let $\Gamma_I, \Gamma_J$ be the $I$-torsion and respectively the $J$-torsion functor. Let $M$ be an $R$-module. Let us recall the construction of the natural maps $\theta^i_{I,J}(M) \colon H^i_I(M) \rt H^i_J(M)$ for all $i \geq 0$:

Let  $\E$ be an injective resolution of $M$. Then note we have a natural morphism of complexes $\theta \colon \Gamma_I(\E) \rt \Gamma_J(\E)$. Taking cohomology  we obtain our natural maps $\theta^i_{I,J}(M)$ for all $i \geq 0$. 

The following result will be used several times.
\begin{lemma}\label{flat}
Let $R \rt S$ be a flat map of Noetherian rings. Let $I, J$  be  ideals of $R$ with $I \supseteq J$. Let $M$ be an $R$-module. Then for all $i \geq 0$ we have
\[
\theta^i_{I,J}(M)\otimes S \cong  \theta^i_{IS,JS}(M\otimes S).
\]
\end{lemma}
\begin{proof}
Let $\E$ be an injective resolution of $M$. Then $\E \otimes S $ is $\Gamma_{KS}$ acyclic for any ideal $K$ of $R$, see \cite[4.1.9]{BSh} . Furthermore we have a natural equivalence of functors $\Gamma_K \otimes S \cong \Gamma_{KS}$, see \cite[4.3.1]{BSh}. Thus we have a commutative diagram of complexes 
\[
\xymatrixrowsep{3pc}
\xymatrixcolsep{2.5pc}
\xymatrix{
    \Gamma_I(\E)\otimes S \ar@{->}[r]^{\theta(M) \otimes S}
     \ar@{->}[d]^{\alpha}
&   \Gamma_J(\E)\otimes S   \ar@{->}[d]^{\beta} 
\\
  \Gamma_{IS}(\E\otimes S) \ar@{->}[r]^{\theta(M\otimes S)}
&     \Gamma_{JS}(\E\otimes S)
}
\]
where $\alpha$ and $\beta $ are isomorphisms of complexes. The result follows.
\end{proof}
\s \label{setup} Let $R$ be a Noetherian ring. 
Let $I \supseteq J$ be ideals in  $R$. We have a natural map
$\theta^i_{I,J} \colon H^i_I(R) \rt H^i_J(R)$ for each $i \geq 0$.  Fix $i_0 \geq 0$.
Set
\[
\C^{i_0}_R =  \left\{ (I,J) \left| I \supset J  \ \text{and} \  \sharp \Ass_R( \image \theta^{i_0}_{I,J}) = \infty  \right.   \right \}.
\]
Here $\sharp S $ denotes the number of elements in a set $S$. 
We partially order $\C^{i_0}_R $ as follows: Set $(I,J) \preceq (I^\prime, J^\prime)$ if $I^\prime \supseteq I$ and $J^\prime \supseteq J$. It is easy to see that every ascending chain in $\C^{i_0}_R $ stabilizes.  If  $\C^{i_0}_R$ is non-empty then its maximal elements has a peculiar property which we now describe:
\begin{lemma}\label{pec}[with hypotheses as in \ref{setup}]
Assume $(I,J) \in \C^{i_0}_R$ is a maximal element. Let $S$ be a multiplicatively closed subset of $R$. If $(S^{-1}I, S^{-1}J) \in \C^{i_0}_{S^{-1}R}$ then
\begin{enumerate}[\rm (1)]
\item
$(S^{-1}I, S^{-1}J)$ is a maximal element in $\C^{i_0}_{S^{-1}R}$.
\item
$S^{-1}I \cap R = I$ and $S^{-1}J \cap R = J$.
\end{enumerate}
\end{lemma}
To prove the Lemma we need the following easy result.
\begin{proposition}\label{easy}
[with hypotheses as in Lemma \ref{pec}]
If $(K,L) \in  \C^{i_0}_{S^{-1}R}$ then $(K\cap R, L \cap R) \in \C^{i_0}_R $.
\end{proposition}
\begin{proof}
Set $K_1 = K \cap R$ and $L_1 = L \cap R$. We note that $S^{-1}K_1 = K$ and $S^{-1}L_1 = L$. By Propositions \ref{flat} and \ref{cong} we get that
\[
S^{-1}\left(\image \theta^{i_0}_{K_1,L_1}(R)\right) \cong \image \theta^{i_0}_{K,L}(S^{-1}R)
\]
It follows that $ \sharp\Ass_R (\image \theta^{i_0}_{K_1,L_1}(R)) = \infty$. So $(K_1,L_1) 
 \in \C^{i_0}_R $.
\end{proof}
We now give
\begin{proof}[Proof of Lemma \ref{pec}]
(1) Suppose $(S^{-1}I, S^{-1}J) \preceq (K,L) $  for some \\ $(K,L) \in  \C^{i_0}_{S^{-1}R}$. Then by Proposition \ref{easy} we get $(K\cap R, L \cap R) \in \C^{i_0}_R $. Notice $(I,J) \preceq (K\cap R, L \cap R)$. By maximality of $(I,J)$ in 
$\C^{i_0}_R $ we get that $I = K\cap R$ and $J = L \cap R$. So $(S^{-1}I, S^{-1}J) = (K,L) $.

(2) Set $K = S^{-1}I$ and $L = S^{-1}J$. Then 
by Proposition \ref{easy} and our hypotheses $(K\cap R, L \cap R) \in \C^{i_0}_R $. Notice
 $(I,J) \preceq (K\cap R, L \cap R)$. So by maximality of $(I,J)$ in 
$\C^{i_0}_R $ we get $I = K\cap R$ and $J = L \cap R$. 
\end{proof}

\s \label{prod} \emph{Product of rings}. Assume $R = R_1\times R_2 \times \cdots \times R_n$. If $R$ is Noetherian then each $R_i$ is Noetherian. Note an ideal $I$ in $R$ is of the form $I_1\times I_2\times \cdots \times I_n$ where $I_j$ is an ideal in $R_j$. Further note that $P$ is a prime in $R$ if and only if $P_i$ is a prime ideal in $R_i$ for some $i$ and $P_j = R_j$ for $j \neq i$. Thus $\Spec(R)$ is a disjoint union of  $\Spec(R_1), \cdots, \Spec(R_n)$.
 
The following result is easy to verify
\begin{proposition}\label{ver}
Let $M_i$ be $R_i$ modules. Then $M = M_1 \times \cdots \times M_n$ is a $R$-module.
Furthermore 
\[
\Ass_R(M) = \bigcup_{i = 1}^{n} \Ass_{R_i}( M_i).
\]
\end{proposition}
\begin{remark}\label{ver-rem}
In the above Proposition a prime $P$ of $R_i$ is identified with the following  prime
of $R$;
\[
R_1\times \cdots \times R_{i-1}\times P \times R_{i+1} \times \cdots \times R_n.
\]
\end{remark}
The following result takes a little work. However it is completely elementary and so we skip it.
\begin{proposition}\label{prod-loc}
[with hypotheses as in \ref{prod}] We have for each $i \geq 0$ an isomorphism
\[
H^i_I(R) \cong H^i_{I_1}(R_1) \times H^i_{I_2}(R_2)\times \cdots  \times H^i_{I_n}(R_n).
\]
\end{proposition}
The following remark will be used often.
\begin{remark}\label{red-domain}
Let $R$ be a regular ring. Let $\{ P_1,\cdots, P_n \}$ be minimal primes of $R$.  Set 
$R_i = R/P_i$.
Then 
$$R \cong R_1 \times R_2 \times \cdots \times R_n \quad \text{see \cite[Exercise 9.11]{Mat}}.$$ 
Notice $R_i$ are regular domains.

If $R$ is excellent then each $R_i$ is excellent. Furthermore if $\dim R = d$ then it is easy to see that $\dim R_i \leq d$ for all $i$ and $\dim R_m = d$ for some $m$.

By \ref{prod-loc} and \ref{ver} it follows that if $I = I_1\times \cdots \times I_n$
then $\Ass H^i_I(R)$ is finite if and only if $\Ass H^i_{I_j}(R_j)$ is finite for all 
$j$. Thus for the questions we are interested it suffices to assume $R$ is a domain.
\end{remark}
\section{The flat extension $R \rt R[[X]]_X$}\label{flat-s}
In our arguments we need to assume that $R$ contains an uncountable field. When this is not the case we consider the flat extension $R \rt R[[X]]_X$.
Set $S = R[[X]]$ and $T = S_X = R[[X]]_X$, i.e., the ring obtained by inverting $X$.

\begin{remark}\label{passage}
(i) If $R$ contains a countable field $K$ then note $K[[X]]$ is a subring of $ S$ and so $K[[X]]_X $ is a subring of $T$. The field $K[[X]]_X$ is uncountable. Thus $T$ contains an uncountable field.

(ii) If $R$ is regular then so is $S$, see \cite[2.2.13]{BH}. Therefore $T$ is also regular.

(i) Let $R$ be an excellent regular  ring of finite dimension containing a field of characteristic zero. As   $S = R[[X]]$ is regular, it is universally catenary. Also $(X) \subseteq \rad S$ and $S/(X) = R$ is excellent. So by \cite{R} we get $S$ is excellent. It follows that $T$ is excellent.

\end{remark}

The following proposition gives information about behavior of primes when we pass from
$R$ to $T$.
\begin{proposition}\label{primes}[with hypotheses as above]
Let $\p$ be a prime in $R$. Then
\begin{enumerate}[\rm (i)]
\item
$\p T$ is a prime in $T$.
\item
$T$ is a faithfully flat extension of $R$.
\item
$\p T \cap R = \p$.
\item
$\height \p T = \height \p$.
\item
Let $P$ be a prime ideal in $T$. If $P \cap R = \p$ then $P = \p T$.
\item
$\dim T = \dim R$.
\end{enumerate}
\end{proposition}
\begin{proof}
(i) Clearly $\p S$ is a prime in $S$. Also $X \notin \p S$. As $\p T$ is localization of $\p S$ we get that it is prime in $T$.

(ii) $T$ is clearly a flat $R$-algebra. Let $\m $ be a maximal ideal of $R$. Then $\m T$ is a prime ideal of $T$. In particular
$\m T \neq T$. So $T$ is faithfully flat \cite[7.2]{Mat}.

(iii) We have 
$$\p T \cap R = \p T \cap S \cap R = \p S \cap R = \p. $$

(iv) By \cite[Theorem 15.1]{Mat} we get
\[
\height \p T = \height \p + \dim T_{\p T}/\p T_{\p T}
\]
As $\p T_{\p T}$ is the maximal ideal in $T_{\p T}$ we get the required result.

(v) Again by \cite[Theorem 15.1]{Mat} we get
\[
\height P = \height \p + \dim T_{P}/\p T_{P}.
\]
Set $\kappa(\p) = A_\p/\p A_\p$, the residue field of $A_\p$. Then note that
$T_{P}/\p T_{P}$ is a localization of $\kappa(\p)[[X]]$ at a multiplicatively closed set containing $X$. It follows that $T_{P}/\p T_{P}$ is a field. So $\height P = \height \p$. Now $P$ contains the prime ideal $\p T$ which by (iii) also has $\height  =  \height \p$. So $P = \p T$.

(vi) This follows easily from (iv) and (v).
\end{proof}

We need Theorem 23.3 from \cite{Mat}.   Unfortunately  there is a typographical error in the statement of Theorem 23.3 in \cite{Mat}.
So we state it here.
\begin{theorem}\label{23}
Let $\varphi \colon A \rt B$ be a homomorphism of Noetherian rings, and let $E$ be an $A$-module and $G$ a $B$-module. Suppose that
$G$ is flat over $A$; then we have the following:
\begin{enumerate}[\rm (i)]
\item
if $\p \in \Spec A$ and $G/\p G \neq 0$ then
\[
 ^a \varphi \left( \Ass_B(G/\p G)  \right) = \Ass_A (G/\p G) = \{  \p \}.
\]
\item
$\displaystyle{\Ass_B(E\otimes_A G) = \bigcup_{\p \in \Ass_A(E) } \Ass_B(G/\p G).}$
\end{enumerate}
\end{theorem}
\begin{remark}
In  \cite{Mat},  $\Ass_A(E\otimes G)$ is typed instead of $\Ass_B(E\otimes G)$. Also note that
$^a \varphi(P) = P \cap A$ for $P \in \Spec B$.
\end{remark}

\s Let $M$ be an $R$-module. Set 
$$\Ass^i_R(M)  = \{ P \mid  P \in \Ass_R(M) \ \text{and} \ \height P = i \}. $$

We now state the main result of this section.
\begin{theorem}
\label{corr}
Let $R$ be a Noetherian ring and let $M$ be an $R$-module. Set $S = R[[X]]$ and $T = S_X$. Then
\begin{enumerate}[\rm (1)]
\item
The mapping defined by
\begin{align*}
\psi \colon \Ass_R(M) &\rt \Ass_T (M\otimes_R T) \\
               \p &\rt \p T
\end{align*}
is a bijection.
\item
 $\psi$ maps $\Ass^i_R(M)$ bijectively to $\Ass^i_T(M\otimes_R T)$.
\end{enumerate}
\end{theorem}
\begin{proof}
(1) By Theorem \ref{23} we get that
\[
\Ass_T(M\otimes_R T) = \bigcup_{\p \in \Ass_R(M)}\Ass_T(T/\p T)   
\]
By  Proposition \ref{primes}(i) we get that $\p T$ is a prime ideal in $T$. So $\Ass_T T/\p T = \{ \p T\}$. Thus
\[
\Ass_T(M\otimes_R T) =  \{ \p T \mid \p \in \Ass_R(M) \}.
\]
So the map $\psi$ is well-defined and surjective. By Proposition \ref{primes}(iii) we get that it is injective.

(2) This follows from (1) and Proposition \ref{primes}(iv).
\end{proof}

An immediate corollary is
\begin{corollary}\label{cc}
[with hypotheses as in Theorem \ref{corr}]
\begin{enumerate}[\rm (1)]
\item
$\Ass_R M$ is an infinite set if and only if $\Ass_T M\otimes_R T$ is an infinite set.
\item
$\Ass^i_R M$ is an infinite set if and only if $\Ass_T^i M\otimes_R T$ is an infinite set.  \qed
\end{enumerate}
\end{corollary}

\section{countable prime avoidance}
\s \label{set1} \emph{Setup:} In this section $R$ is a Noetherian ring containing an uncountable field $K$. We describe a construction which we will use often.
The essential ingredient is the following  well-known countable avoidance. We give a proof due to lack of a suitable reference.
\begin{lemma}\label{countable}
(with hypotheses as in \ref{set1}) Let $\{I_n\}_{n \geq 1}$ be  ideals in $R$ and let $J$ be another ideal in $R$.
If $J \subseteq \bigcup_{n \geq 1}I_n$ then $J \subseteq I_m$ for some $m$.
\end{lemma}
\begin{proof}
Let $J = (x_1,\ldots,x_c)$. Let $V = Kx_1 + K  x_2 + \cdots + K x_c$. Then $V$ is a finite dimensional $K$-vector-space.
Also clearly $V = \bigcup_{n \geq 1} V \cap I_n$. As $K$ is an uncountable field we get that $V\cap I_m = V$ for some $m$. Thus 
$x_i \in I_m$ for all $i$. So $J \subseteq I_m$.
\end{proof}
We now give our
\s \label{const-K} \textbf{Construction:} (with hypotheses as in \ref{set1}).  Assume $\{ \p_n \}_{n \geq 1}$ is a sequence of primes in $R$. Also assume that $\p_i \nsubseteq \p_j$ for $i \neq j$. Consider the multiplicatively closed set 
\[
S = R \setminus \bigcup_{n \geq 1} \p_n.
\]
Set $T = S^{-1}R$. 
The following result gives information about primes in $T$.
\begin{proposition}
\label{prop-const}(with hypotheses as in \ref{const-K}). 
We have
\begin{enumerate}[\rm (1)]
\item
If $\p T$ is a prime in $T$, where $\p$ is a prime in $R$, then $\p \subseteq \p_n$ for some $n$.
\item
$\p_iT \nsubseteq \p_j T$ for $i \neq j$.
\item
$\p_n T$ are distinct maximal ideals of $T$.
\item
$\mSpec(T) = \{ \p_n T\}_{n \geq 1}$.
 \end{enumerate}
\end{proposition}
\begin{proof}
(1) We have $\p \cap S = \emptyset$. So $\p \subseteq  \bigcup_{n \geq 1} \p_n$. By \ref{countable} we get that $\p \subseteq \p_n$ for some $n$.

(2) If $\p_iT \subseteq \p_j T$ for some $i \neq j$ then intersecting with $R$ we get $\p_i \subseteq \p_j$ for some $i \neq j$; a contradiction.

(3) Let $\p_n T \subseteq  P$ for some prime $P$ of $T$. Say $P = \p T$ where $\p$ is a prime in $R$. By (1) we get that $\p T \subseteq \p_m T$ for some $m$. By (2) we get $m = n$. So $P = \p_n T$. Thus $\p_n T$ is a maximal ideal in $T$. That they are all distinct follows from (2).

(4) This follows from (1) and (3).
\end{proof}
We will need the following intersection result:
\begin{proposition}\label{prop-const-2}(with hypotheses as in \ref{const-K}). 
Let $\Lambda$ be any subset of $\{ \p_n \}_{n \geq 1}$ ( possibly infinite). Set 
$$ U = \bigcap_{\p \in \Lambda} \p \quad \text{and} \ V = \bigcap_{\p \in \Lambda} \p T. $$
Then $U T = V$.
\end{proposition}
\begin{proof}
Clearly $U T \subseteq V$. Let $\xi = a/s \in V$. 
Let $\p \in \Lambda$. As $\xi \in \p T$ we get that 
$\xi = r/s_1$ where $r \in \p$. It follows that there exists $s' \in S$ such that $s's_1 a \in \p$. As $s's_1 \notin \p$ we get that $a \in \p$. Thus $a \in U$. Therefore $\xi \in U T$. Thus $V = UT$.
\end{proof}
\section{Proof of Theorem \ref{main}}
In this section we prove Theorem \ref{main}.

\begin{proof}[Proof of Theorem \ref{main}]
Suppose $\C^{i_0}_R$ is non-empty for some $i_0$. Let $(I,J)$ be a maximal element in $\C^{i_0}_R$. It follows from Proposition \ref{count-ass} we get that 
$\Ass_R (\image \theta^{i_0}_{I,J})$ is a countably infinite set. As $\dim R$ is finite we can choose an infinite subset 
$\{ \p_n \}_{n \geq 1}$ of $\Ass_R (\image \theta^{i_0}_{I,J})$ 
such that $\height \p_i = \height \p_j$ for all $i,j$. Clearly the primes $\{\p_n \}_{n\geq 1}$ are mutually incomparable. Set 
\[
S = R \setminus \bigcup_{n \geq 1}\p_n \quad \text{and} \ A = S^{-1}R.
\]
By \ref{prop-const} we get $\mSpec(A) = \{ \p_n A\}_{n \geq 1}$. Furthermore by construction $\height \p_n A  = \height \p_n$ is constant. Also by Lemma \ref{flat} we get that \\
$\Ass_A (\image \theta^{i_0}_{IA,JA}) \supseteq \mSpec(A)$. By \ref{pec} we also get that $(IA, JA)$ is a maximal element of $\C^{i_0}_A$.

Suppose if possible $g = \height IA  < i_0$. 

\textit{Claim 1:} $\Ass_A H^{i_0}_{IA}(A)$ is an infinite set.

Suppose if possible $\Ass_A H^{i_0}_{IA}(A) = \{ Q_1, Q_2, \ldots, Q_s \}$ is a finite set. Notice $\height Q_i \geq i_0 > g$ for all $i$. Note that $IA$ is a radical ideal. Let $IA = P_1\cap P_2 \cap \cdots \cap P_r$ for some primes $P_j$. Say $\height P_1 = g$. Choose $x_i \in Q_i \setminus P_1$. Set $x = x_1x_2\cdots x_s$. Then $x \in Q_i$ for all $i$. Also $x \notin P_1$. So $x \notin IA$.

We have an exact sequence
\[
\cdots \rt H^{i}_{IA + (x)}(A) \xrightarrow{\theta^{i}_{IA + (x), IA}} H^i_{IA}(A) \rt 
\left(H^i_{IA}(A)\right)_x  \rt \cdots
\]
As $x \in Q_i$ for all $i$ we get that $\left(H^{i_0}_{IA}(A)\right)_x = 0$. Thus
$\theta^{i_0}_{IA + (x), IA}$ is surjective. 
Notice
\[
 \theta^{i_0}_{IA + (x), JA} = \theta^{i_0}_{IA, JA} \circ \theta^{i_0}_{IA + (x), IA}.
\]
As $\theta^{i_0}_{IA + (x), IA}$ is surjective 
it follows that  $\Ass_A(\image \theta^{i_0}_{IA + (x), JA})$ is an infinite set. So $(IA + (x), JA) \in \C^{i_0}_A$. Also $x \notin IA$. This contradicts  the maximality of $(IA, JA)$ in $\C^{i_0}_A$.  Thus $\Ass_A H^{i_0}_{IA}(A)$ is an infinite set.

Now let $\Lambda$ be an \emph{infinite} subset of $\Ass_A (H^{i_0}_{IA}(A))$.
Set 
\[
W_{\Lambda} = \bigcap_{P \in \Lambda}P.
\]
Clearly $IA \subseteq W_{\Lambda}$.

\textit{Claim 2:} $IA = W_{\Lambda}$. \\
Suppose if possible there exist $x \in W_{\Lambda} \setminus IA$. We have an exact  sequence
\[
\cdots \rt H^{i_0}_{IA + (x)}(A) \xrightarrow{\theta^{i_0}_{IA + (x), IA}} H^{i_0}_{IA}(A) \xrightarrow{\pi_{i_0}}
\left(H^{i_0}_{IA}(A)\right)_x  \rt \cdots
\]
For $P \in \Lambda$ let
 $P = (0 \colon a_P)$ for some $a_P \in H^{i_0}_{IA}(A)$. Clearly $\pi_{i_0}(a_P) = 0$  for all $P \in \Lambda$. It follows that $\Lambda \subseteq \Ass_A(\image(\theta^{i_0}_{IA + (x), IA})$. Thus $(IA + (x), IA) \in \C^{i_0}_A$. Also $x \notin IA$. This contradicts  the maximality of $(IA, JA)$ in $\C^{i_0}_A$. Thus Claim 2 is true.
 
 Let 
 $$L_j = \{ P \mid P \in \Ass_A(H^{i_0}_{IA}(A)) \ \text{and} \ \height P = j \}.$$
 As $\dim A$ is finite we get by Claim 1 that $L_j$ is infinite for some $j$. Choose $j_0$ to be the maximum $j$ with $L_j$ infinite. Set
 \[
 \Lambda = L_{j_0} \setminus \{ \p \mid \p \ \text{a minimal prime of} \ IA \ \text{and} \ \height \p = j_0 \}.
 \]
 Clearly $\Lambda$ is an infinite set. Set 
 $$T = A \setminus \bigcup_{P \in \Lambda}P \quad \text{and} \ B = T^{-1}A. $$
 By \ref{prop-const} we get that 
 $$\mSpec(B) = \{ PB \mid P \in \Lambda \}.$$
 It follows from Claim-2 and \ref{prop-const-2} that
 \[
 IB = \bigcap_{P \in \Lambda}PB.
 \]
 Thus $IB$ is the Jacobson radical of $B$.  Note we also get that $\Ass_B H^{i_0}_{IB}(B)$ contains $\mSpec(B)$. 
 
 $R$ is excellent. So $A$ is excellent and hence $B$ is excellent. Therefore $\Reg(B/IB)$ is an open set. As $IB$ is a radical ideal we get that $\Reg(B/IB)$ is \emph{non-empty}.  Since $\rad(B/IB) = 0$ it follows that there exists a maximal ideal $\m$ of $B$ with $(B/IB)_\m$ a regular local ring.  As $B$ is a regular ring, $B_\m$ is a regular ring. Also note that by our construction $IB_\m \neq \m B_\m$. As $B_\m/IB_\m$ is regular we get that $IB_\m$ is generated by part of a regular system of parameters. Say $IB_\m = (x_1,\ldots, x_c)$ and $c < \dim B_\m$. We now note that
 \[
 H^j_{IB_\m}(B_\m) = 0 \ \text{for} \ j \neq c \quad \text{and} \ \Ass H^c_{IB_\m}(B_\m) = \{(x_1,\ldots,x_c) \}.
 \]
 It follows that $\m \notin \Ass_B H^{i_0}_{IB}(B)$; a contradiction. 
 Thus $\height IA = i_0$.
 
 Let $IA = P_1\cap P_2 \cap \cdots \cap P_s \cap Q_1 \cap Q_2 \cap \cdots \cap Q_l$,
 where $\height P_j = i_0$ for all $j$ and $\height Q_j \geq i_0 + 1$ for all $j$. Set
 $K = P_1 \cap \cdots \cap P_s$. It is well-known that the natural map
 \[
 \theta^{i_0}_{K, IA} \colon H^{i_0}_K(A) \rt H^{i_0}_{IA}(A)  \quad \text{is an isomorphism.}
 \]
 It follows that $(K, JA) \in \C^{i_0}_A$. By maximality of $(IA,JA)$ in  $\C^{i_0}_A$
 we get that $K = IA$.
 
 Now assume $\Ass_R (H^r_L(R))$ is a finite set for all $r < i_0$ and for all ideals $L$ of $R$. In particular $\Ass_R(H^{i_0 -1}_J(R))$ is a finite set. So
 $\Ass_A(H^{i_0 -1}_{JA}(A))$ is a finite set. 
 Suppose if possible $\height JA < i_0 - 1$. Let
 \[
 \Ass_A(H^{i_0 -1}_{JA}(A)) = \{ Q_1, Q_2, \cdots, Q_s \}.
 \]
 Notice $\height Q_j \geq i_0 -1$ for all $j$.
 Let $JA = P_1 \cap P_2 \cdots \cap P_r$ where $\height P_1 = \height JA < i_0 -1$.
 Choose $y_j \in Q_j \setminus P_1$. Also choose $x \in I \setminus P_1$. Set
 $t = y_1y_2\cdots y_s x$. Then $t \in Q_j$ for all $j$. We note that 
 $$\left(H^{i_0 -1}_{JA}(A)\right)_t = 0.$$
 Also note that as $t \notin P_1$ we get $t \notin JA$.
 
 Thus we have a commutative diagram
 \[
\xymatrixrowsep{3pc}
\xymatrixcolsep{2.5pc}
\xymatrix{
0  \ar@{->}[r]
  &   H^{i_0 }_{IA + (t)}(A) \ar@{->}[r]^\cong
     \ar@{->}[d]^{\theta^{i_0}_{IA, JA + (t) }}
&     H^{i_0 }_{IA}(A) \ar@{->}[d]^{\theta^{i_0}_{IA, JA  }}
            \ar@{->}[r]
&0
\\
0 \ar@{->}[r]
 & H^{i_0 }_{JA + (t)}(A) \ar@{->}[r]^{\theta^{i_0}_{JA + (t), JA }}
&  H^{i_0 }_{JA}(A)   \ar@{->}[r]^\pi
& \left(H^{i_0 }_{JA}(A) \right)_t
}
\]
We note that as $t \in I$, we get $t \in \rad(A)$. Let $\m$ be a maximal ideal in $A$ and let $(0 \colon a_\m) = \m$ where $a_\m \in \image \theta^{i_0}_{IA, JA  } $.  Clearly $\pi(a_\m) = 0$. Thus it follows that $\m \in \Ass_A (H^{i_0 }_{JA + (t)}(A))$. A simple diagram chase shows that in-fact $\m \in \Ass_A ( \image \theta^{i_0}_{IA, JA + (t) })$. Thus
$(IA , JA + (t)) \in \C^{i_0}_A$.  This contradicts  the maximality of $(IA, JA)$ in $\C^{i_0}_A$.  Therefore $\height JA = i_0 - 1$.
\end{proof}

As an immediate consequence we get
\begin{proof}[Proof of Corollary \ref{cor-main}]
$(i) \implies (ii)$. Obvious.

$(ii) \implies (i)$. By Theorem \ref{main} it follows that Lyubeznik's conjecture holds for all regular domains of dimension $ \leq d$ containing an uncountable  field 
of characteristic zero. By results in section \ref{flat-s} it follows that Lyubeznik's conjecture holds for all regular domains of dimension $ \leq d$ containing a field 
of characteristic zero. By \ref{red-domain} the result holds for all regular rings of dimension $ \leq d$ containing a  field 
of characteristic zero.
\end{proof}

\section{Proof of Theorem \ref{main-T}}
In this section we prove Theorem \ref{main-T}. We will need the following result (see \cite[Proof of Theorem 31.1]{Mat}).
\begin{lemma}\label{int}
Let $\{ Q_n \}_{n \geq 1}$ be an infinite family of primes in a Noetherian ring $T$. Let $P$ be another prime ideal in $T$ with $P \subseteq Q_n$ for all $n$. Suppose $\height(Q_n/P) = 1$ for all $n$. Then
\[
\bigcap_{n\geq 1} Q_n = P.
\]
\end{lemma}
We also need Corollary \ref{local}. As the techniques to prove it are totally different we postpone the proof of Corollary \ref{local} to the next section. We now give:
\begin{proof}[Proof of Theorem \ref{main-T}]
By \ref{red-domain} we may assume that $R$ is a domain. By results in section \ref{flat-s}, we may assume that $R$ contains an uncountable field of characteristic zero. 

Suppose if possible for some ideal $K$ of $R$ we have $\Ass_R (H^{d-1}_K(R))$ is an infinite set. So $(K,K) \in \C^{d-1}_R$. Thus $\C^{d-1}_R \neq \emptyset$. Let $(I,J)$
be a maximal element in $\C^{d-1}_R$. We now do the construction as in Theorem \ref{main}.  Then in the ring $A = S^{-1}R$ we have 
\begin{enumerate}[\rm(a)]
\item
$(IA, JA)$ is a maximal element in $\C^{d-1}_ A$.
\item
$\height IA = d-1$.
\item
$IA = P_1\cap P_2 \cap \cdots \cap P_r$ where $P_i$ is a prime in $A$ of height $d-1$.
\item
$\Ass_A (\image \theta^{d-1}_{IA, JA}) \supseteq \mSpec(A)$.
\item
$\mSpec(A)$ is a countably infinite set.
\end{enumerate}
We now note that $\dim A = d$. This is so since for any Noetherian ring $T$ of dimension $n$ and an ideal  $L$ of $T$ we have $\Ass_T (H^n_L(T))$ is a finite set.  Thus $\dim A \neq d -1$. Furthermore by Grothendieck vanishing theorem it is not possible that $\dim A < d -1$.  Thus  $\dim A = d$. Again by Theorem \ref{main} we get  
 that $\height \m = d$ for all maximal ideals of $A$. We note that $IA \subseteq \rad(A)$. 

\textit{Claim 1:} There exists a localization $B$ of $A$ such that
\begin{enumerate}
\item
$\dim B = d$.
\item
$\height \m = d$ for all maximal ideals $\m$ of $B$.
\item
$\mSpec(B)$ is a countably infinite set.
\item
$\Ass_B (\image \theta^{d-1}_{IB, JB}) \supseteq \mSpec(B)$.
\item
$IB$ is a prime ideal of height $d-1$.
\item
$IB = \rad(B)$.
\end{enumerate}
To prove the claim we recall that $IA = P_1\cap P_2\cap \cdots \cap P_r$ where $P_i$ is a prime ideal in $A$ of height $d-1$.
We consider two cases.

Case 1: $r = 1$. \\
Then $IA = P_1$ is a prime ideal of height $d-1$. Also $P_1 \subseteq \m$ for all maximal ideals $\m$ of $A$. As $\height \m = d$ for all maximal ideals of $A$ and as $\mSpec(A)$ is a countably infinite set, by Lemma \ref{int} we get $IA = P_1 = \rad(A)$. Thus we can take $B$ to be $A$.

Case 2: $r \geq 2$. \\
Consider the sets
\[
Y_i = \{ \m \mid \m \in \mSpec(A) \ \text{and} \ \m \supseteq P_i \}.
\]
As $IA \subseteq \rad(A)$ we get $Y_1 \cup Y_2 \cup \cdots \cup Y_r = \mSpec(A)$. So there exists $i$ such that $Y_{i}$  is an infinite set. After relabeling we may assume $i = 1$.
Set 
\[
T_1 = A \setminus \bigcup_{\m \in Y_1} \m \quad \text{and} \quad A_1 = T_1^{-1}A.
\]
Let $IA_1 = Q_1\cap Q_2\cap \cdots \cap Q_s$ be an \textit{irredundant} primary decomposition of $IA_1$ with $Q_1 = P_1A_1$. (Note as $IA_1$ is a radical ideal all  $Q_i$ are prime ideals.  We note that $\height Q_1 = d-1$. Furthermore 
$\height \n = d$ for each maximal ideal of $A_1$.  Also by Proposition \ref{prop-const}, $\sharp \mSpec(A_1) = \sharp Y_1 =  \infty$. As $Q_1 \subseteq \m$ for each maximal ideal of $A_1$, by Lemma \ref{int} we get 
$Q_1 = \rad(A_1)$. 

We also note that 
$\Ass_{A_1} (\image \theta^{d-1}_{IA_1, JA_1}) \supseteq \mSpec(A_1)$ and $IA_1 \subseteq \rad(A_1)$. For $j \geq 2$, consider the sets
\[
Y_j^\prime = \{ \m \mid \m \in \mSpec(A_1) \ \text{and} \ \m \supseteq Q_j \}.
\]
\textit{Claim 2:} 
$Y_j^\prime$ is a finite set for all $j \geq 2$. \\
Suppose if possible $Y_j$ is an infinite set for some $j$. Then by Lemma \ref{int} we get 
\[
Q_j = \bigcap_{\m \in Y_j^\prime}\m \supseteq \rad(A_1) = Q_1
\]
This contradicts the fact that $Q_1\cap \cdots \cap Q_s$ is an \emph{irredundant}
primary decomposition of $IA_1$. Thus Claim 2 is proved.

Now set 
$$\Lambda = \mSpec(A_1) \setminus \bigcup_{j \geq 2} Y_j^\prime, \quad T_2 = A_1\setminus \bigcup_{\m \in \Lambda} \m \quad \text{and} \quad B = T_2^{-1}A_1.  
$$
It is easy to prove that $B$ satisfies all the assertions in Claim 1.

 We now note that $B$ is excellent. So $\Reg(B/IB)$ is  an open set. As $IB$ is a prime ideal we get that $\Reg(B/IB)$ is \emph{non-empty}. Since $\rad(B/IB) = 0$ it follows that there exists a maximal ideal $\m$ of $B$ with $(B/IB)_\m$ a regular local ring.  As $B$ is a regular ring, $B_\m$ is a regular ring. Also note that  $\height IB_\m  = d -1$ and $\height \m B_\m = d$. As $B_\m/IB_\m$ is regular we get that $IB_\m$ is generated by part of a regular system of parameters. Say $IB_\m = (x_1,\ldots, x_{d-1})$ and say $\m B_\m = (x_1,\cdots, x_{d-1},x_d)$. Let $k$ be the residue field of $B_\m$. Set $C = \widehat{B_\m}$. Note $C = k[[x_1,\cdots,x_d]]$ and $IC = (x_1,\ldots,x_{d-1})$.  Let $\n = (x_1,\cdots, x_d)C$ be the maximal ideal of $C$.
 By our assumption $\n \in \Ass_C (\image \theta^{d-1}_{IC, JC}) $.
However this contradicts Corollary \ref{local}. 
Thus $\Ass_R (H^{d-1}_K(R))$ is a finite set for all ideals $K$ of $R$. 
\end{proof}

\section{Simplicity of a local cohomology module}
The reference for this section is \cite{B} and \cite{Lyu-1}.
Let $\Oc = K[[X_1,\cdots, X_n]]$ where $K$ is a field of characteristic zero. Let $D$ be the ring of $K$-linear differential operators on $D$.  By the work of Lyubeznik it is known that if $I$ is an ideal in $\Oc$ then $H^i_I(\Oc)$ are finitely generated $D$-module for all $i \geq 0$. The main goal of this section is to prove that a certain local cohomology  $\Oc$-module is simple as 
a $D$-module.

\s Consider the $\sum$ filtration 
where
$\sum_k = \{ Q \in D \mid Q = \sum_{|\alpha| \leq k} q_{\alpha}(\bX)\partial^\alpha \}$, 
is the set of differential operators of order $\leq k$. 
The associated graded ring $\gr D = \sum_0 \oplus \sum_1/\sum_0 \oplus \cdots $ is isomorphic to the polynomial ring $\Oc[\zeta_1,\cdots, \zeta_n]$ where $\zeta_i$ is the image of $\partial_i$ in $\sum_1/\sum_0$. 

\s Let $M$ be a finitely generated $D$-module. We consider  filtrations $\Fc$
of $K$-linear subspaces of $M$ with the property that $\Fc_i = 0$ for $i < 0$, $\Fc_i \subseteq \Fc_{i+1}$, $M = \bigcup_{i \geq 0} \Fc_i$ and $\sum_i \Fc_j \subseteq \Fc_{i+j}$. Note that $\gr_{\Fc} M = \bigoplus_{i\geq 0} \Fc_{i}/\Fc_{i-1}$ is a $\gr D$-module. We say $\Fc$ is a \textit{good filtration} of $M$ if $\gr_\Fc M$ is finitely generated as a $\gr D$-module. We note that every finitely generated $D$-module $M$ has a good filtration.

\s Let $\Mr$ be the unique graded maximal ideal of $\gr D$. Let $\Fc$ be a good filtration on $M$ then note that
$\dim \gr_\Fc M  = \dim (\gr_\Fc M)_\Mr$, see \cite[1.5.8]{BH}. Let $e(\gr_\Fc M )$ be the multiplicity of $(\gr_\Fc M)_\Mr$ with respect to the maximal ideal $\Mr (\gr D)_\Mr$ of the regular local ring $(\gr D)_\Mr$.  Let $\Fc, \Gc$ be two good filtrations on $M$. Then by \cite[Lemma 6.2, Chapter 2]{B} we get that
\[
\dim \gr_\Fc M = \dim \gr_\Gc M \quad \text{and} \quad e(\gr_\Fc M ) = e(\gr_\Gc M). 
\]
Thus if $\Fc$ is a good filtration on $M$ then we can set
\[
\dim M  = \dim \gr_\Fc M \quad \text{and} \quad e_\Mr(M) = e(\gr_\Fc M ). 
\]
\s It is well-known that if $M$ is a non-zero $D$-module then $\dim M \geq n$. If $\dim M = n$ or if $M = 0$ then we say $M$ is a \textit{holonomic} $D$-module. By the work of Lyubeznik it is known that if $I$ is an ideal in $\Oc$ then $H^i_I(\Oc)$ is a holonomic $D$-module for all $i \geq 0$.

We need the following Lemma and its corollary:
\begin{lemma}
Let $M$ be a holonomic $D$-module. If $M \neq 0$ and $M$ is not simple then 
$e_\Mr(M) \geq 2$.
\end{lemma}
\begin{proof}
As $M$ is not simple it has a proper non-zero submodule $K$. Set $C = M/K$. Then $K, C$ are \emph{non-zero} holonomic $D$-modules.

Let $\Fc$ be a good filtration on $M$. Set $\ov{\Fc} = $ quotient filtration on $C$ and let $\Gc = \{ \Fc_n \cap K \}_{n \geq 0}$ be the induced filtration on $K$. Then we have an exact sequence of graded $\gr D$-modules,
\begin{equation*}
0 \rt \gr_\Gc K \rt \gr_\Fc M \rt \gr_{\ov{\Fc}} C \rt 0. \tag{*}
\end{equation*}
Thus $\gr_\Gc K$ and $\gr_{\ov{\Fc}} C$ are finitely generated $\gr D$ modules. So
 $\ov{\Fc}$ is a good filtration of $C$ and  $\Gc$ is a good filtration of $K$.
 
 All the modules in equation (*) have dimension $n$. Computing multiplicites we get
 \[
 e_\Mr(M) =   e_\Mr(K) +  e_\Mr(C) \geq 2.
 \]
\end{proof}
As an immediate corollary we obtain
\begin{corollary}\label{m-1}
Let $M$ be a holonomic $D$-module. If $e_\Mr(M) = 1$ then $M$ is a simple $D$-module.
\end{corollary}
The main result of this section is:
\begin{lemma}\label{m-loc}
Let $\Oc = K[[X_1,\ldots, X_n]]$ and let $P_g = (X_1,\ldots, X_g)$ for $g \geq 1$. Then $e_\Mr(H^g_{P_g}(\Oc)) = 1$. Thus
$H^g_{P_g}(\Oc)$ is a simple $D$-module.
\end{lemma}
\begin{proof}
We first consider the case when $g = n$. So $P_n = \m = (X_1,\ldots, X_n)$. In this case it is well-known that 
$H^n_\m(\Oc) \cong D/D\m = K[\ov{\partial_1},\ldots, \ov{\partial_n}]$ as $D$-modules. Let $\Fc_i = \{ \sum_{|\alpha| \leq i}a_\alpha \ov{\partial}^\alpha \mid a_\alpha \in K \}$.  Let $Q \in \sum_i$ be a differential operator of order $\leq i$. Then notice $Q$ can also be written as $Q = \sum_{|\alpha| \leq i}\partial^\alpha a_\alpha$.  Set $a_\alpha = c_\alpha + t_\alpha$ where $a_\alpha \in K$ and $t_\alpha \in \m$. Then notice that the image of $Q$ in $D/D\m$ is  $= \sum_{|\alpha| \leq i} c_\alpha\ov{\partial}^\alpha$.
Thus $\Fc$ is the quotient filtration of $\sum$. 
 Therefore $\Fc$ is a good filtration on $D/D\m$. Also note that
$$\gr_\Fc D/D\m =  \gr D / \m \gr D = K[\zeta_1,\ldots,\zeta_n].$$
Clearly $e_\Mr(D/D\m) = 1$. Also note that $X_i\Fc_{\nu} \subseteq \Fc_{\nu-1}$ for all $\nu \geq 0$.

We now consider the case when $g < n$. Set $S = K[[X_1,\ldots,X_g]]$ and $\n = (X_1,\ldots,X_g)$. Set $D^\prime$ be the ring of $K$-linear differential operators on $S$. Let $M = H^g_\n(S) = D^\prime/D^\prime \n$. Set
$N = H^g_{P_g}(\Oc) = M\otimes_S \Oc$. We note that that $D$-module structure on $N$ is given by  $\partial_1,\ldots,\partial_g$ acting on $M$ and $\partial_i$ acting on $\Oc$ for $i > g$. Also note that for $r \in  \Oc$, $m \in M$ and $t \in \Oc$ we have
$r\cdot(m\otimes t) = m\otimes rt$. Let $\Fc$ be the $D^\prime$ filtration on $M$ as discussed earlier.
Set
\[
\Omega_\nu = \{ \sum_{\text{finite sum}} m_\alpha\otimes r_\alpha \mid m_\alpha \in \Fc_\nu \ \text{and} \ r_\alpha \in \Oc \}. 
\]
Let $\xi = a\partial_1^{a_1}\cdots\partial_g^{a_g}\cdots \partial_n^{a_n} \in \sum_i$. So $\sum_{k=1}^{n}a_k \leq i$. 
In particular $\sum_{k = 1}^{g}a_k \leq i$. Let $m\otimes r \in \Omega_\nu $.
 Then
\[
\xi\cdot(m\otimes r) = \left(  (\partial_1^{a_1} \cdots \partial_g^{a_g} m \right) \otimes \left( a(\partial_{g+1}^{a_{g+1}} \cdots \partial_n^{a_n}r) \right) \in \Omega_{\nu + i}.
\]
Thus $\Omega$ is a filtration on $N$ compatible with the $\sum$-filtration on $D$. Therefore $\gr_\Omega N$ is a $\gr D$-module.

We first assert that $\gr_\Omega N$ is generated by $\ov{1\otimes 1} \in \Omega_0$. Let $\xi = m\otimes r \in \Omega_\nu$. Then note that if $m = \sum_{\alpha}a_\alpha \ov{\partial}^\alpha$ then $m = Q\cdot 1$ where $Q = \sum_{\alpha}a_\alpha {\partial}^\alpha$. Thus $\xi = rQ\cdot(1\otimes 1)$. Therefore we have an obvious surjective  map $\phi \colon \gr D \rt \gr_\Omega N$ which takes $1$ to $\ov{1\otimes 1}$. Thus $\Omega$ is a good filtration on $N$. 

 Let $\xi = [m\otimes r] \in \Omega_\nu/\Omega_{\nu-1}$. For $i \leq g$ we have $X_i m \in \Fc_{\nu-1}$. Thus 
 $X_i (m\otimes r) = m\otimes X_ir = X_im\otimes r \in \Omega_{\nu-1}$. Therefore $X_i\xi =0$ for all $i \leq g$. 
Also notice that for $i > g$ we have $\zeta_i\xi = [m\otimes \partial_i(r)]$. But $m\otimes \partial_i r \in \Omega_\nu$. As $\degree \zeta_i = 1$ we get that $\zeta_i \xi = 0$.
Thus $\phi$ factors to a surjective map $\ov{\phi} \colon \Oc/(X_1,\ldots,X_g)[\zeta_1,\ldots,\zeta_g] \rt \gr_\Omega N$.

 As $\gr_\Omega N$ has dimension $n$ it follows that $\ov{\phi}$ is in fact an isomorphism. Thus $\gr_\Omega N  \cong \Oc/(X_1,\ldots,X_g)[\zeta_1,\ldots,\zeta_g]$ and clearly it has multiplicity one.
\end{proof}

We need the following result in the proof of our main result.
\begin{corollary}\label{local}
Let $\Oc = K[[X_1,\ldots,X_n]]$ and let $P = (X_1,\ldots, X_g)$ where $g < n$. Set $\m = (X_1,\ldots,X_n)$. Let $J \subseteq P$.
Let $\theta^g \colon H^g_P(\Oc) \rt H^g_J(\Oc)$ be the natural map. Then $\Ass \image \theta^g$ is either empty or equal to $\{ P\}$. In particular $\m \notin \Ass \image \theta^g$.
\end{corollary}
\begin{proof}
We note that any injective $D$-module is also an injective $\Oc$-module. Thus we can take an injective resolution of $\Oc$ as a $D$-module and note that computing $\theta^i$ with this resolution we get that $\theta^i \colon H^i_P(\Oc) \rt H^i_J(\Oc)$ is $D$-linear. By \ref{m-loc} we get that $N = H^g_P(\Oc)$ is a simple $D$-module. Thus $\image \theta^g  \cong N$ or it is zero. The result follows. 
\end{proof}

\section{Small dimensions}
In this section we prove several elementary results regarding local cohomology modules of regular rings of dimension $\leq 4$. 
We note that the results for dimension $\leq 3$ are probably well-known to the experts. However as we are unable to find a reference we give a proof for these cases too. 
 
We first prove the following general result.
\begin{proposition}\label{HL}
Let $R$ be a regular ring of dimension $d$. Let $I$ be an ideal in $R$. Then
$$\Ass_R (H^d_I(R)) = \left\{ P  \left|  P \in \Min R/I \ \text{and} \ \height P = d   \right.   \right \}.  $$ 
\end{proposition}
\begin{proof}
We may assume $I$ is a radical ideal.
Suppose $P \in \Ass_R (H^d_I(R))$. Then as $\height P \geq d = \dim R$ we get that $P$ is a maximal ideal of $R$. Note $PR_P \in \Ass_{R_P} (H^d_{IR_P}(R_P))$ and thus
$P\widehat{R_P} \in \Ass_{\widehat{R_P}} (H^d_{I\widehat{R_P}}(\widehat{R_P}))$.
 
Notice $\widehat{R_P}$ is a domain and so by  Hartshorne-Lichtenbaum theorem, cf.
\cite[14.1]{a7} we get that $I\widehat{R_P}$ is a zero-dimensional ideal in $\widehat{R_P}$. It follows that $IR_P$ is zero-dimensional ideal in $R_P$. As $I R_P$ is a radical ideal we get that $IR_P = PR_P$. The result follows.
\end{proof}
If $M$ is an $R$-module then set 
\[
\Ass_R^i(M) = \{ P \mid P \in \Ass M \ \text{and} \ \height P = i \}.
\]
As an easy corollary to Proposition \ref{HL} we get
\begin{corollary}\label{HL-C}
Let $R$ be a regular ring of dimension $d$. Let $I$ be an ideal in $R$. Then
\[
\bigcup_{i \geq 0} \Ass_R^i(H^i_I(R))  = \Min R/I.
\]
\end{corollary}
\begin{proof}
We may assume $I$ is a radical ideal.
Let $I = P_1\cap \cdots \cap P_r$. Notice $IR_{P_i} = P_iR_{P_i}$ for each $i$ and $R_{P_i}$ is a regular local ring of height $c_i = \height P_i$. Notice
$P_i \in  \Ass_R^{c_i}(H^{c_i}_I(R))$. Thus
\[
\Min R/I \subseteq \bigcup_{i \geq 0} \Ass_R^i(H^i_I(R)).
\]
Conversely let $P \in \Ass_R^i(H^i_I(R))$. We localize at $P$. Then $R_P$ is a regular local ring of dimension $i$. The result now follows from Proposition \ref{HL}.
\end{proof}
\begin{remark}\label{htg}
If $I$ is an ideal of height $g$ then it is well-known that
\[
\Ass H^g_I(R) = \{ P \mid P \supseteq I, \height P = g \}.
\]
\end{remark}
We now prove:
\begin{lemma}\label{ht1}
Let $R$ be a regular ring and let $I$ be an ideal of height one. Then $\Ass_R H^i_I(R)$ is finite for $i = 1,2$
\end{lemma}
\begin{proof}
For $i = 1$ the result follows from Remark \ref{htg}.

 For $i = 2$ we first consider the case when  $I$ has a primary decomposition $I = Q_1\cap \cdots \cap Q_r$ where $\height \sqrt{Q_i} = 1$ for all $i$. We claim that $H^j_I(R) = 0$ for all $j \geq 2$. Suppose this is not true. Let $P \in \Ass H^j_I(R)$ for some $j \geq 2$. We localize at $P$. We now note that $IR_P$ is a principal ideal, see \cite[Exercise 2.2.28]{BH}. So $H^s_{IR_P}(R_P) = 0$ for all $s \geq 2$, a contradiction. Therefore $H^j_I(R) = 0$ for all $j \geq 2$. Thus our assertion holds in this special case. 

Now let $I$ be a general ideal of height one. Then $I = J \cap K$ where $K$ has height $\geq 2$ and $J$ is an ideal of height one. Furthermore $J$ is of the special kind discussed above. So $H^j_J(R) = 0$ for $j \geq 2$. By Mayer-Vietoris sequence, cf., \cite[15.1]{a7},  and noting that $\height (J + K) \geq 3$ we have an exact sequence
\[
0 \rt H^2_K(R) \rt H^2_I(R) \rt H^3_{J+K}(R).
\]
As $\height K \geq 2$ and $\height (J+K) \geq 3$ the result follows from Remark \ref{htg}.
\end{proof}
An easy consequence of the above results is the following:
\begin{corollary}\label{d3}
Let $R$ be a regular ring of dimension $d  \leq 3$. Then for any ideal $I$ we have
$\Ass H^i_I(R)$ is a finite set for all $i \geq 0$.
\end{corollary}
\begin{proof}
By Remark \ref{red-domain} we may assume that $R$ is a domain.
We have nothing to show for $d = 0$. The assertion for $d = 1$ follows from Proposition \ref{HL}. For $d = 2$ the result follows from Lemma \ref{ht1} and 
Remark \ref{htg}.

Now consider the case when $d = 3$.  If $\height I = 1$ then the result follows from Lemma \ref{ht1} and  Proposition \ref{HL}. If $\height I = 2$ then the result follows from Remark \ref{htg} and  Proposition \ref{HL}. If $\height I = 3$ then the result follows from 
Proposition \ref{HL}.
\end{proof}

We now give a proof of Corollary \ref{d4}.
\begin{proof}
By Corollary \ref{d3} we may assume $\dim R = 4$.
By the results of \cite{HuSh}, the result holds when $\charr K = p > 0$. Thus we may assume that $\charr K = 0$.  By Remark \ref{red-domain} we may assume that $R$ is a domain. 

If $\height I = 1$ then the result follows from Lemma \ref{ht1}, Theorem \ref{main-T} and  Proposition \ref{HL}. If $\height I = 2$ then the result follows from Remark \ref{htg}, Theorem \ref{main-T} and  Proposition \ref{HL}. If $\height I = 3$ then the result follows from Remark \ref{htg} and
Proposition \ref{HL}. If $\height I = 4$ then the result follows from 
Proposition \ref{HL}.
\end{proof}

\section{Proof of Corollary \ref{main-app} and Proposition \ref{ogus}}
In this final section we give an application of our Theorem \ref{main-T}. Throughout $R$ will denote a regular ring of dimension $d$ containing a field of characteristic zero. We first give
\begin{proof}[Proof of Corollary \ref{main-app}]
Suppose if possible
\[
\bigcup_{i \geq 0} \Ass_R^{i+1}(H^i_I(R))  \quad \text{is an  infinite set}.
\]
As $H^i_I(R) = 0 $ for $i > d$ we get that $\Ass_R^{i+1}(H^i_I(R))$ is infinite for some $i \leq d$. By Corollary \ref{cc} we may assume that $R$ contains an uncountable field. Suppose $\Ass_R^{i+1}(H^i_I(R)) = \{ \p_n \}_{n \geq 1}$. Consider the ring
\[
A = S^{-1}R \quad \text{where} \ S = R \setminus \bigcup_{n \geq 1}\p_n.
\] 
Then $A$ is a regular ring of dimension $i+1$. Furthermore $\Ass_A H^i_{IA}(A)$ is an infinite set. This contradicts Theorem \ref{main-T}.
\end{proof}

Recall if $(A,\m)$ is a local ring 
then $\Spec^\circ(A) = \Spec(A) \setminus \{ \m \}$ considered as a subspace of $\Spec(A)$.
\begin{definition}
Let $(A,\m)$ be a local ring and  let $I$ be an ideal in $A$.  We say $\Spec^\circ(A/I)$ is \textit{absolutely connected} if for every flat local map $(A,\m) \rt (B,\n)$ with $\m B = \n$ and $B/\n$ algebraically closed,  $\Spec^\circ(B/IB)$ is connected 
\end{definition}
\begin{remark}\label{c-abs}
It is easy to see that if $\Spec^\circ(A/I)$ is absolutely connected then it is connected.
\end{remark}
We now give
\begin{proof}[Proof of Proposition \ref{ogus}]
First assume that $P \supseteq I$,  $\height P = c  \geq g + 2$ and $\Spec^\circ(R_P/I_P)$ is not absolutely connected. So there is a flat extension $(B,\n)$ of $R_P$ such that $B$ is complete, $PB = \n$, $B/\n$ algebraically closed and $\Spec^\circ(B/IB)$ is disconnected. We note that $B$ is a complete regular ring of dimension $c$. Also note that as $\dim R_P/I_P \geq 2$ we have that $\dim B/IB \geq 2$. By the result of Ogus \cite[2.11]{O} we have that $H^{c-1}_{IB}(B) \neq 0$. So $H^{c-1}_{IR_P}(R_P) \neq 0$.
Therefore $P \in \Supp H^{c-1}_{I}(R)$. We claim that $P$ is a minimal prime of
$H^{c-1}_{I}(R)$. Suppose there exists $Q \in \Supp H^{c-1}_{I}(R)$ and $Q \subsetneq P$. Then $\height Q \leq c - 1$. By Grothendieck vanishing theorem we have $\height Q \geq c -1$. So $\height Q = c - 1$. By our assumption we have that $\dim R_Q/IR_Q \geq 1$. So $\dim \widehat{R_Q}/I\widehat{R_Q} \geq 1$. By Hartshorne-Lichtenbaum theorem we have $H^{c-1}_{I\widehat{R_Q}}(\widehat{R_Q}) =0$. Thus $(H^{c-1}_I(R))_Q = H^{c-1}_{IR_Q}(R_Q) = 0$, a contradiction. Therefore $P$ is a minimal prime of $H^{c-1}_{I}(R)$ and so belongs to $\Ass^{c}(H^{c-1}_I(R))$.

Conversely assume $P \in \Ass^{c}(H^{c-1}_I(R))$. So $H^{c-1}_{IR_P}(R_P) \neq 0$. By Remark \ref{htg} we get that $c - 1 \geq g + 1$. So $c \geq g + 2$.  Let $R_P \rt B$ be any  local flat extension with $(B,\n)$ complete, $PB = \n$ and $B/\n$ algebraically closed. We note that by our assumptions $\dim R_P/I R_P \geq 2$. As $R_P/IR_P \rt B/IB$ is flat local map with fiber a field  we have that $\dim B/IB =  \dim R_P/I R_P \geq 2$. Also by faithful flatness $H^{c-1}_{IB}(B) \neq 0$. Thus again by the same result of Ogus, $\Spec^\circ(B/IB)$ is disconnected. Therefore $\Spec^\circ(R_P/I_P)$ is not absolutely connected. 
\end{proof}
\begin{remark}
The above result is also true in characteristic $p$ with the same proof. The reason is that Ogus Theorem is true in characteristic $p$, see \cite[Theorem III.5.5]{PS}.
\end{remark}

\end{document}